\documentclass[reqno]{amsart}

\usepackage{amsmath,amssymb,amsthm,amsfonts,mathrsfs}
\usepackage{fullpage}
\usepackage{xcolor}
\usepackage{graphicx}
\usepackage{listings}

\definecolor{codegreen}{rgb}{0,0.6,0}
\definecolor{codegray}{rgb}{0.5,0.5,0.5}
\definecolor{codepurple}{rgb}{0.58,0,0.82}
\definecolor{backcolour}{rgb}{0.95,0.95,0.92}

\lstdefinestyle{mystyle}{
  backgroundcolor=\color{backcolour},   commentstyle=\color{codegreen},
  keywordstyle=\color{magenta},
  numberstyle=\tiny\color{codegray},
  stringstyle=\color{codepurple},
  basicstyle=\ttfamily\footnotesize,
  breakatwhitespace=false,         
  breaklines=true,                 
  captionpos=b,                    
  keepspaces=true,                 
  numbers=left,                    
  numbersep=5pt,                  
  showspaces=false,                
  showstringspaces=false,
  showtabs=false,                  
  tabsize=2
}

\usepackage{listings}
\usepackage[colorlinks=true,
linkcolor=blue,
anchorcolor=blue,
citecolor=red
]{hyperref}
\allowdisplaybreaks 
\newtheorem{theorem}{Theorem}[section]

\newtheorem{corollary}[theorem]{Corollary}

\newtheorem*{conjecture*}{Conjecture}
\theoremstyle{definition}

\theoremstyle{remark}

\newtheorem*{remark*}{remark}

\usepackage{colonequals}

\author{Runbo Li}
\address{International Curriculum Center, The High School Affiliated to Renmin University of China, Beijing, China}
\email{runbo.li.carey@gmail.com}

\makeatletter
\@namedef{subjclassname@2020}{\textup{}2020 Mathematics Subject Classification}
\makeatother

\title[]{A note on variants of Buchstab's identity}

\begin{document}
	
\begin{abstract}
The author proves variants of Buchstab's identity on sieve functions, refining the previous work on new iteration rules of Brady. The main tool used in the proof is a special form of combinatorial identities related to the binomial coefficients. As a by--product, the author obtains better inequalities of $F_{\kappa}(s)$ and $f_{\kappa}(s)$ for dimensions $\kappa > 1$.
\end{abstract}

\maketitle

\tableofcontents

\section{Introduction}
Let $\mathcal{A}$ be a set of numbers, $\mathcal{A}_d = \{a: a d \in \mathcal{A}\}$ and $S(\mathcal{A}, z)=\sum_{\substack{a \in \mathcal{A} \\ (a, \prod_{p<z} p)=1}} 1$. Suppose that $\kappa, z, y$ are such that for every squarefree integer $d$, all of whose prime factors are less than $z$, we have
\begin{equation}
\left| \left| \mathcal{A}_d \right| - \kappa^{\omega(d)} \frac{y}{d} \right| \leqslant 1.
\end{equation}
Suppose that $y = z^s$ and define $F_{\kappa}(s)$ and $f_{\kappa}(s)$ by
\begin{equation}
(1+o(1))f_{\kappa}(s)y \prod_{p < z}\left(1-\frac{\kappa}{p}\right) \leqslant S(\mathcal{A}, z) \leqslant (1+o(1))F_{\kappa}(s)y \prod_{p < z}\left(1-\frac{\kappa}{p}\right)
\end{equation}
with $f_{\kappa}(s)$ as large as possible and $F_{\kappa}(s)$ as small as possible, given that (2) holds for all choices of $\mathcal{A}$ satisfying (1). Selberg \cite{SelbergBOOK} has shown that $F_{\kappa}(s)$ and $f_{\kappa}(s)$ are continuous, monotone, and computable for $s > 1$, and that they tend to $1$ exponentially as $s$ goes to infinity.

When $\kappa \leqslant 1$, the optimal estimates for $F_{\kappa}(s)$ and $f_{\kappa}(s)$ arise from Buchstab's identity
\begin{equation}
S\left(\mathcal{A}, z\right) = S\left(\mathcal{A}, w\right) - \sum_{w \leqslant p < z} S\left(\mathcal{A}_{p}, p\right)
\end{equation}
for $w \leqslant z$. Simply let $w = 2$, this becomes
\begin{equation}
S\left(\mathcal{A}, z\right) = \left|\mathcal{A}\right| - \sum_{p < z} S\left(\mathcal{A}_{p}, p\right).
\end{equation}
This leads to the inequalities
\begin{align}
s^{\kappa} F_{\kappa}(s) \leqslant&\ s^{\kappa} - \kappa \int_{t>s} t^{\kappa -1} \left( f_{\kappa}(t-1) -1\right) d t, \\
s^{\kappa} f_{\kappa}(s) \geqslant&\ s^{\kappa} - \kappa \int_{t>s} t^{\kappa -1} \left( F_{\kappa}(t-1) -1\right) d t.
\end{align}
Infinite iteration of these inequalities leads to the $\beta$--sieve.

However, there are better estimates for $F_{\kappa}(s)$ and $f_{\kappa}(s)$ when $\kappa > 1$. Taking Selberg's upper bound sieve as a starting point and using similar iteration rules, Diamond, Halberstam and Richert \cite{DHR} developed their DHR--sieve.

In 2017, Brady mentioned and proved lots of new sieve iteration rules in his PhD thesis. One of his simplest upper bound sieve is
\begin{equation}
\nonumber S\left(\mathcal{A}, z\right) \leqslant S\left(\mathcal{A}, w\right) - \frac{2}{3} \sum_{w \leqslant p_1 < z} S\left(\mathcal{A}_{p_1}, w\right) + \frac{1}{3} \sum_{w \leqslant p_2 < p_1 < z} S\left(\mathcal{A}_{p_1 p_2}, w\right).
\end{equation}
He proved this inequality using a combinatorial identity
\begin{equation}
1 - \frac{2}{3} n + \frac{1}{3} \left(\genfrac{}{}{0pt}{}{n}{2}\right) = \left(1-\frac{n}{2}\right) \left(1-\frac{n}{3}\right).
\end{equation}
Clearly, this leads to an inequality of $F_{\kappa}(s)$:
\begin{equation}
s^{\kappa} F_{\kappa}(s) \leqslant t^{\kappa} F_{\kappa}(t) - \frac{2}{3} \int_{\frac{1}{t}}^{\frac{1}{s}} \frac{t^{\kappa} f_{\kappa}(t(1-x_1))}{x_1} d x_1 + \frac{1}{3} \int_{\frac{1}{t}}^{\frac{1}{s}} \int_{\frac{1}{t}}^{x_1} \frac{t^{\kappa} F_{\kappa}(t(1-x_1-x_2))}{x_1 x_2} d x_2 d x_1.
\end{equation}
In this note, we further develop his method and prove a series of generalized iteration rules.

\section{Upper bound iteration}
We first prove a simple upper bound iteration, which is a direct generalization of [\cite{BradyThesis}, Theorem 34].
\begin{theorem}\label{UB}
For any odd positive integer $k$ and $w \leqslant z$, we have
\begin{align}
\nonumber S\left(\mathcal{A}, z\right) \leqslant&\ S\left(\mathcal{A}, w\right) - \frac{k-1}{k} \sum_{w \leqslant p_1 < z} S\left(\mathcal{A}_{p_1}, w\right) + \frac{k-2}{k} \sum_{w \leqslant p_2 < p_1 < z} S\left(\mathcal{A}_{p_1 p_2}, w\right) \\
\nonumber &- \frac{k-3}{k} \sum_{w \leqslant p_3 < p_2 < p_1 < z} S\left(\mathcal{A}_{p_1 p_2 p_3}, w\right) + \cdots \\
\nonumber &- \frac{2}{k} \sum_{w \leqslant p_{k-2} < \cdots < p_2 < p_1 < z} S\left(\mathcal{A}_{p_1 p_2 \cdots p_{k-2} }, w\right)\\
\nonumber &+ \frac{1}{k} \sum_{w \leqslant p_{k-1} < \cdots < p_2 < p_1 < z} S\left(\mathcal{A}_{p_1 p_2 \cdots p_{k-1} }, w\right).
\end{align}
\end{theorem}
\begin{proof}
We follow the essential steps in the proof of [\cite{BradyThesis}, Theorem 34]. Let $a \in \mathcal{A}$. If $a$ has any prime factor below $w$, then both quantities are clearly zero. Assume that $a$ has no prime factors below $w$ and has exactly $n$ prime factors between $w$ and $z$. If $n = 0$ then both sides count $a$ once. Thus we only need to show that for any integer $n \geqslant 1$ we have
\begin{equation}
0 \leqslant 1 - \frac{k-1}{k} n + \frac{k-2}{k} \left(\genfrac{}{}{0pt}{}{n}{2}\right) - \frac{k-3}{k} \left(\genfrac{}{}{0pt}{}{n}{3}\right) + \cdots + \frac{1}{k} \left(\genfrac{}{}{0pt}{}{n}{k-1}\right).
\end{equation}
Note that we have the following identity
\begin{equation}
1 - \frac{k-1}{k} n + \frac{k-2}{k} \left(\genfrac{}{}{0pt}{}{n}{2}\right) - \frac{k-3}{k} \left(\genfrac{}{}{0pt}{}{n}{3}\right) + \cdots + \frac{1}{k} \left(\genfrac{}{}{0pt}{}{n}{k-1}\right) = \left(1-\frac{n}{2}\right) \left(1-\frac{n}{3}\right) \cdots \left(1-\frac{n}{k}\right)
\end{equation}
and the right hand side of (), which has even number of terms, is clearly $\geqslant 0$, Theorem~\ref{UB} is proved. Note that [\cite{BradyThesis}, Theorem 34] is just Theorem~\ref{UB} with $k=3$.
\end{proof}

\begin{corollary}
For any odd positive integer $k$ and real $2 \leqslant s \leqslant t$, we have
\begin{align}
\nonumber s^{\kappa} F_{\kappa}(s) \leqslant&\ t^{\kappa} F_{\kappa}(t) - \frac{k-1}{k} \int_{\frac{1}{t}}^{\frac{1}{s}} \frac{t^{\kappa} f_{\kappa}(t(1-x_1))}{x_1} d x_1 + \frac{k-2}{k} \int_{\frac{1}{t}}^{\frac{1}{s}} \int_{\frac{1}{t}}^{x_1} \frac{t^{\kappa} F_{\kappa}(t(1-x_1-x_2))}{x_1 x_2} d x_2 d x_1 \\
\nonumber &- \frac{k-3}{k} \int_{\frac{1}{t}}^{\frac{1}{s}} \int_{\frac{1}{t}}^{x_1} \int_{\frac{1}{t}}^{x_2} \frac{t^{\kappa} f_{\kappa}(t(1-x_1-x_2-x_3))}{x_1 x_2 x_3} d x_3 d x_2 d x_1 + \cdots \\
\nonumber &- \frac{2}{k} \int_{\frac{1}{t}}^{\frac{1}{s}} \int_{\frac{1}{t}}^{x_1} \cdots \int_{\frac{1}{t}}^{x_{k-3}} \frac{t^{\kappa} f_{\kappa}(t(1-x_1-x_2-\cdots -x_{k-2}))}{x_1 x_2 \cdots x_{k-2}} d x_{k-2} \cdots d x_2 d x_1 \\
\nonumber &+ \frac{1}{k} \int_{\frac{1}{t}}^{\frac{1}{s}} \int_{\frac{1}{t}}^{x_1} \cdots \int_{\frac{1}{t}}^{x_{k-2}} \frac{t^{\kappa} F_{\kappa}(t(1-x_1-x_2-\cdots -x_{k-1}))}{x_1 x_2 \cdots x_{k-1}} d x_{k-1} \cdots d x_2 d x_1.
\end{align}
\end{corollary}

However, we can use more flexible parameters to get more variants of this iteration. Before stating the next result, we first define
\begin{equation}
\mathcal{U} = \left\{(x_1, x_2): x_1, x_2 \in (0,1] \cup [2,3] \cup \cdots \cup [k-1, k] \text{ with all odd }k,\ |x_1 - x_2| \leqslant 1 \right\}.
\end{equation}

\begin{theorem}\label{UB2D}
For any $m_1, m_2$ such that $(m_1, m_2) \in \mathcal{U}$ and $w \leqslant z$, we have
\begin{equation}
\nonumber S\left(\mathcal{A}, z\right) \leqslant S\left(\mathcal{A}, w\right) - \frac{m_1 + m_2 - 1}{m_1 m_2} \sum_{w \leqslant p_1 < z} S\left(\mathcal{A}_{p_1}, w\right) + \frac{2}{m_1 m_2} \sum_{w \leqslant p_2 < p_1 < z} S\left(\mathcal{A}_{p_1 p_2}, w\right).
\end{equation}
\end{theorem}
\begin{proof}
Again, we use the essentially same arguments as the proof of Theorem~\ref{UB}. Let $a \in \mathcal{A}$. If $a$ has any prime factor below $w$, then both quantities are clearly zero. Assume that $a$ has no prime factors below $w$ and has exactly $n$ prime factors between $w$ and $z$. If $n = 0$ then both sides count $a$ once. Thus we only need to show that for any integer $n \geqslant 1$ we have
\begin{equation}
0 \leqslant 1 - \frac{m_1 + m_2 - 1}{m_1 m_2} n + \frac{2}{m_1 m_2} \left(\genfrac{}{}{0pt}{}{n}{2}\right).
\end{equation}
By the following identity
\begin{equation}
1 - \frac{m_1 + m_2 - 1}{m_1 m_2} n + \frac{2}{m_1 m_2} \left(\genfrac{}{}{0pt}{}{n}{2}\right) = \left(1-\frac{n}{m_1}\right) \left(1-\frac{n}{m_2}\right)
\end{equation}
and $(x_1, x_2) \in \mathcal{U}$, which means that , we know that the right--hand side of (4) is clearly $\geqslant 0$, Theorem~\ref{UB2D} is proved. Note that [\cite{BradyThesis}, Theorem 34] is just Theorem~\ref{UB2D} with $m_1=2$ and $m_2=3$.
\end{proof}

\begin{corollary}
For any $m_1, m_2$ such that $(m_1, m_2) \in \mathcal{U}$ and real $2 \leqslant s \leqslant t$, we have
\begin{align}
\nonumber s^{\kappa} F_{\kappa}(s) \leqslant&\ t^{\kappa} F_{\kappa}(t) - \frac{m_1 + m_2 - 1}{m_1 m_2} \int_{\frac{1}{t}}^{\frac{1}{s}} \frac{t^{\kappa} f_{\kappa}(t(1-x_1))}{x_1} d x_1 \\
\nonumber &+ \frac{2}{m_1 m_2} \int_{\frac{1}{t}}^{\frac{1}{s}} \int_{\frac{1}{t}}^{x_1} \frac{t^{\kappa} F_{\kappa}(t(1-x_1-x_2))}{x_1 x_2} d x_2 d x_1.
\end{align}
\end{corollary}

Using the same method but with more parameters, we can get lots of upper bound iterations of this type. For the sake of simplicity, we write
$$
M^{r}_{k} = \sum_{1 \leqslant i_1 < i_2 < \cdots < i_r \leqslant k} m_{i_1} m_{i_2} \cdots m_{i_r}.
$$
\begin{theorem}\label{UB4D}
For any $m_1, m_2, m_3, m_4$ such that $(m_1, m_2) \in \mathcal{U}$, $(m_3, m_4) \in \mathcal{U}$ and $w \leqslant z$, we have
\begin{align}
\nonumber S\left(\mathcal{A}, z\right) \leqslant&\ S\left(\mathcal{A}, w\right) - \frac{M^{3}_{4} - M^{2}_{4} + M^{1}_{4} - 1}{m_1 m_2 m_3 m_4} \sum_{w \leqslant p_1 < z} S\left(\mathcal{A}_{p_1}, w\right) + \frac{2 \left(M^{2}_{4} - 3 M^{1}_{4} + 7\right)}{m_1 m_2 m_3 m_4} \sum_{w \leqslant p_2 < p_1 < z} S\left(\mathcal{A}_{p_1 p_2}, w\right) \\
\nonumber &- \frac{6 \left(M^{1}_{4} - 6\right)}{m_1 m_2 m_3 m_4} \sum_{w \leqslant p_3 < p_2 < p_1 < z} S\left(\mathcal{A}_{p_1 p_2 p_3}, w\right) + \frac{24}{m_1 m_2 m_3 m_4} \sum_{w \leqslant p_4 < p_3 < p_2 < p_1 < z} S\left(\mathcal{A}_{p_1 p_2 p_3 p_4}, w\right).
\end{align}
\end{theorem}

\begin{corollary}
For any $m_1, m_2, m_3, m_4$ such that $(m_1, m_2) \in \mathcal{U}$, $(m_3, m_4) \in \mathcal{U}$ and real $2 \leqslant s \leqslant t$, we have
\begin{align}
\nonumber s^{\kappa} F_{\kappa}(s) \leqslant&\ t^{\kappa} F_{\kappa}(t) - \frac{M^{3}_{4} - M^{2}_{4} + M^{1}_{4} - 1}{m_1 m_2 m_3 m_4} \int_{\frac{1}{t}}^{\frac{1}{s}} \frac{t^{\kappa} f_{\kappa}(t(1-x_1))}{x_1} d x_1 \\
\nonumber &+ \frac{2 \left(M^{2}_{4} - 3 M^{1}_{4} + 7\right)}{m_1 m_2 m_3 m_4} \int_{\frac{1}{t}}^{\frac{1}{s}} \int_{\frac{1}{t}}^{x_1} \frac{t^{\kappa} F_{\kappa}(t(1-x_1-x_2))}{x_1 x_2} d x_2 d x_1 \\
\nonumber &- \frac{6 \left(M^{1}_{4} - 6\right)}{m_1 m_2 m_3 m_4} \int_{\frac{1}{t}}^{\frac{1}{s}} \int_{\frac{1}{t}}^{x_1} \int_{\frac{1}{t}}^{x_2} \frac{t^{\kappa} f_{\kappa}(t(1-x_1-x_2-x_3))}{x_1 x_2 x_3} d x_3 d x_2 d x_1 \\
\nonumber &+ \frac{24}{m_1 m_2 m_3 m_4} \int_{\frac{1}{t}}^{\frac{1}{s}} \int_{\frac{1}{t}}^{x_1} \int_{\frac{1}{t}}^{x_2} \int_{\frac{1}{t}}^{x_3} \frac{t^{\kappa} F_{\kappa}(t(1-x_1-x_2-x_3-x_4))}{x_1 x_2 x_3 x_4} d x_4 d x_3 d x_2 d x_1.
\end{align}
\end{corollary}

\begin{theorem}\label{UB6D}
For any $m_1, m_2, m_3, m_4, m_5, m_6$ such that $(m_1, m_2) \in \mathcal{U}$, $(m_3, m_4) \in \mathcal{U}$, $(m_5, m_6) \in \mathcal{U}$ and $w \leqslant z$, we have
\begin{align}
\nonumber S\left(\mathcal{A}, z\right) \leqslant&\ S\left(\mathcal{A}, w\right) - \frac{M^{5}_{6} - M^{4}_{6} + M^{3}_{6} - M^{2}_{6} + M^{1}_{6} - 1}{m_1 m_2 m_3 m_4 m_5 m_6} \sum_{w \leqslant p_1 < z} S\left(\mathcal{A}_{p_1}, w\right) \\
\nonumber &+ \frac{2 \left(M^{4}_{6} - 3 M^{3}_{6} + 7 M^{2}_{6} - 15 M^{1}_{6} + 31\right)}{m_1 m_2 m_3 m_4 m_5 m_6} \sum_{w \leqslant p_2 < p_1 < z} S\left(\mathcal{A}_{p_1 p_2}, w\right) \\
\nonumber &- \frac{6 \left(M^{3}_{6} - 6 M^{2}_{6} + 25 M^{1}_{6} - 90\right)}{m_1 m_2 m_3 m_4 m_5 m_6} \sum_{w \leqslant p_3 < p_2 < p_1 < z} S\left(\mathcal{A}_{p_1 p_2 p_3}, w\right) \\
\nonumber &+ \frac{24 \left( M^{2}_{6} - 10 M^{1}_{6} + 65\right)}{m_1 m_2 m_3 m_4 m_5 m_6} \sum_{w \leqslant p_4 < p_3 < p_2 < p_1 < z} S\left(\mathcal{A}_{p_1 p_2 p_3 p_4}, w\right) \\
\nonumber &- \frac{120 \left(M^{1}_{6} - 15\right)}{m_1 m_2 m_3 m_4 m_5 m_6} \sum_{w \leqslant p_5 < p_4 < p_3 < p_2 < p_1 < z} S\left(\mathcal{A}_{p_1 p_2 p_3 p_4 p_5}, w\right) \\
\nonumber &+ \frac{720}{m_1 m_2 m_3 m_4 m_5 m_6} \sum_{w \leqslant p_6 < p_5 < p_4 < p_3 < p_2 < p_1 < z} S\left(\mathcal{A}_{p_1 p_2 p_3 p_4 p_5 p_6}, w\right).
\end{align}
\end{theorem}

\begin{corollary}
For any $m_1, m_2, m_3, m_4, m_5, m_6$ such that $(m_1, m_2) \in \mathcal{U}$, $(m_3, m_4) \in \mathcal{U}$, $(m_5, m_6) \in \mathcal{U}$ and real $2 \leqslant s \leqslant t$, we have
\begin{align}
\nonumber s^{\kappa} F_{\kappa}(s) \leqslant&\ t^{\kappa} F_{\kappa}(t) - \frac{M^{5}_{6} - M^{4}_{6} + M^{3}_{6} - M^{2}_{6} + M^{1}_{6} - 1}{m_1 m_2 m_3 m_4 m_5 m_6} \int_{\frac{1}{t}}^{\frac{1}{s}} \frac{t^{\kappa} f_{\kappa}(t(1-x_1))}{x_1} d x_1 \\
\nonumber &+ \frac{2 \left(M^{4}_{6} - 3 M^{3}_{6} + 7 M^{2}_{6} - 15 M^{1}_{6} + 31\right)}{m_1 m_2 m_3 m_4 m_5 m_6} \int_{\frac{1}{t}}^{\frac{1}{s}} \int_{\frac{1}{t}}^{x_1} \frac{t^{\kappa} F_{\kappa}(t(1-x_1-x_2))}{x_1 x_2} d x_2 d x_1 \\
\nonumber &- \frac{6 \left(M^{3}_{6} - 6 M^{2}_{6} + 25 M^{1}_{6} - 90\right)}{m_1 m_2 m_3 m_4 m_5 m_6} \int_{\frac{1}{t}}^{\frac{1}{s}} \int_{\frac{1}{t}}^{x_1} \int_{\frac{1}{t}}^{x_2} \frac{t^{\kappa} f_{\kappa}(t(1-x_1-x_2-x_3))}{x_1 x_2 x_3} d x_3 d x_2 d x_1 \\
\nonumber &+ \frac{24 \left( M^{2}_{6} - 10 M^{1}_{6} + 65\right)}{m_1 m_2 m_3 m_4 m_5 m_6} \int_{\frac{1}{t}}^{\frac{1}{s}} \int_{\frac{1}{t}}^{x_1} \int_{\frac{1}{t}}^{x_2} \int_{\frac{1}{t}}^{x_3} \frac{t^{\kappa} F_{\kappa}(t(1-x_1-x_2-x_3-x_4))}{x_1 x_2 x_3 x_4} d x_4 d x_3 d x_2 d x_1 \\
\nonumber &- \frac{120 \left(M^{1}_{6} - 15\right)}{m_1 m_2 m_3 m_4 m_5 m_6} \times \\
\nonumber & \quad \int_{\frac{1}{t}}^{\frac{1}{s}} \int_{\frac{1}{t}}^{x_1} \int_{\frac{1}{t}}^{x_2} \int_{\frac{1}{t}}^{x_3} \int_{\frac{1}{t}}^{x_4} \frac{t^{\kappa} F_{\kappa}(t(1-x_1-x_2-x_3-x_4-x_5))}{x_1 x_2 x_3 x_4 x_5} d x_5 d x_4 d x_3 d x_2 d x_1 \\
\nonumber &+ \frac{720}{m_1 m_2 m_3 m_4 m_5 m_6} \times \\
\nonumber & \quad \int_{\frac{1}{t}}^{\frac{1}{s}} \int_{\frac{1}{t}}^{x_1} \int_{\frac{1}{t}}^{x_2} \int_{\frac{1}{t}}^{x_3} \int_{\frac{1}{t}}^{x_4} \int_{\frac{1}{t}}^{x_5} \frac{t^{\kappa} F_{\kappa}(t(1-x_1-x_2-x_3-x_4-x_5-x_6))}{x_1 x_2 x_3 x_4 x_5 x_6} d x_6 d x_5 d x_4 d x_3 d x_2 d x_1.
\end{align}
\end{corollary}

\section{Lower bound iteration}
In this section we shall use a similar method to prove corresponding lower bound iterations.
\begin{theorem}\label{LB3D}
For any $0 < m_0 \leqslant 1$ and $m_1, m_2$ such that $(m_1, m_2) \in \mathcal{U}$ and $w \leqslant z$, we have
\begin{align}
\nonumber S\left(\mathcal{A}, z\right) \geqslant&\ S\left(\mathcal{A}, w\right) - \frac{m_0 m_1 + m_0 m_2 + m_1 m_2 -m_0-m_1-m_2 + 1}{m_0 m_1 m_2} \sum_{w \leqslant p_1 < z} S\left(\mathcal{A}_{p_1}, w\right) \\
\nonumber &+ \frac{2\left(m_0+m_1+m_2-3\right)}{m_0 m_1 m_2} \sum_{w \leqslant p_2 < p_1 < z} S\left(\mathcal{A}_{p_1 p_2}, w\right) \\
\nonumber &- \frac{6}{m_0 m_1 m_2} \sum_{w \leqslant p_3 < p_2 < p_1 < z} S\left(\mathcal{A}_{p_1 p_2 p_3}, w\right).
\end{align}
\end{theorem}
\begin{proof}
By the same arguments as in the proof of Theorem , we only need to show that for any integer $n \geqslant 1$ we have
\begin{equation}
0 \geqslant 1 - \frac{m_0 m_1 + m_0 m_2 + m_1 m_2 -m_0-m_1-m_2 + 1}{m_0 m_1 m_2} n + \frac{2\left(m_0+m_1+m_2-3\right)}{m_0 m_1 m_2} \left(\genfrac{}{}{0pt}{}{n}{2}\right) - \frac{6}{m_0 m_1 m_2} \left(\genfrac{}{}{0pt}{}{n}{3}\right).
\end{equation}
Here, we have the identity
\begin{align}
\nonumber &\ 1 - \frac{m_0 m_1 + m_0 m_2 + m_1 m_2 -m_0-m_1-m_2 + 1}{m_0 m_1 m_2} n + \frac{2\left(m_0+m_1+m_2-3\right)}{m_0 m_1 m_2} \left(\genfrac{}{}{0pt}{}{n}{2}\right) - \frac{6}{m_0 m_1 m_2} \left(\genfrac{}{}{0pt}{}{n}{3}\right) \\
=&\ \left(1-\frac{n}{m_0}\right) \left(1-\frac{n}{m_1}\right) \left(1-\frac{n}{m_2}\right).
\end{align}
One can easily check that for $0 < m_0 \leqslant 1$ and $m_1, m_2$ such that $(m_1, m_2) \in \mathcal{U}$, the right--hand side is zero or negative for any positive integer $n$. Hence Theorem~\ref{LB3D} is proved.
\end{proof}

\begin{corollary}
For any $0 < m_0 \leqslant 1$ and $m_1, m_2$ such that $(m_1, m_2) \in \mathcal{U}$ and $3 \leqslant s \leqslant t$, we have
\begin{align}
\nonumber s^{\kappa} f_{\kappa}(s) \geqslant&\ t^{\kappa} f_{\kappa}(t) - \frac{m_0 m_1 + m_0 m_2 + m_1 m_2 -m_0-m_1-m_2 + 1}{m_0 m_1 m_2} \int_{\frac{1}{t}}^{\frac{1}{s}} \frac{t^{\kappa} F_{\kappa}(t(1-x_1))}{x_1} d x_1 \\
\nonumber &+ \frac{2\left(m_0+m_1+m_2-3\right)}{m_0 m_1 m_2} \int_{\frac{1}{t}}^{\frac{1}{s}} \int_{\frac{1}{t}}^{x_1} \frac{t^{\kappa} f_{\kappa}(t(1-x_1-x_2))}{x_1 x_2} d x_2 d x_1 \\
\nonumber &- \frac{6}{m_0 m_1 m_2} \int_{\frac{1}{t}}^{\frac{1}{s}} \int_{\frac{1}{t}}^{x_1} \int_{\frac{1}{t}}^{x_2} \frac{t^{\kappa} F_{\kappa}(t(1-x_1-x_2-x_3))}{x_1 x_2 x_3} d x_3 d x_2 d x_1.
\end{align}
\end{corollary}

Again, for the sake of simplicity, we write
$$
N^{r}_{k} = \sum_{0 \leqslant i_1 < i_2 < \cdots < i_r \leqslant k-1} m_{i_1} m_{i_2} \cdots m_{i_r}.
$$

\begin{theorem}\label{LB5D}
For any $0 < m_0 \leqslant 1$, $m_1, m_2, m_3, m_4$ such that $(m_1, m_2) \in \mathcal{U}$, $(m_3, m_4) \in \mathcal{U}$ and $w \leqslant z$, we have
\begin{align}
\nonumber S\left(\mathcal{A}, z\right) \geqslant&\ S\left(\mathcal{A}, w\right) - \frac{N^{4}_{5} - N^{3}_{5} + N^{2}_{5} - N^{1}_{5} + 1}{m_0 m_1 m_2 m_3 m_4} \sum_{w \leqslant p_1 < z} S\left(\mathcal{A}_{p_1}, w\right) \\
\nonumber &+ \frac{2\left(N^{3}_{5} - 3 N^{2}_{5} + 7 N^{1}_{5} - 15\right)}{m_0 m_1 m_2 m_3 m_4} \sum_{w \leqslant p_2 < p_1 < z} S\left(\mathcal{A}_{p_1 p_2}, w\right) \\
\nonumber &- \frac{6\left(N^{2}_{5} - 6 N^{1}_{5} + 25 \right)}{m_0 m_1 m_2 m_3 m_4} \sum_{w \leqslant p_3 < p_2 < p_1 < z} S\left(\mathcal{A}_{p_1 p_2 p_3}, w\right) \\
\nonumber &+ \frac{24\left(N^{1}_{5} - 10 \right)}{m_0 m_1 m_2 m_3 m_4} \sum_{w \leqslant p_4 < p_3 < p_2 < p_1 < z} S\left(\mathcal{A}_{p_1 p_2 p_3 p_4}, w\right) \\
\nonumber &- \frac{120}{m_0 m_1 m_2 m_3 m_4} \sum_{w \leqslant p_5 < p_4 < p_3 < p_2 < p_1 < z} S\left(\mathcal{A}_{p_1 p_2 p_3 p_4 p_5}, w\right).
\end{align}
\end{theorem}

\begin{corollary}
For any $0 < m_0 \leqslant 1$, $m_1, m_2, m_3, m_4$ such that $(m_1, m_2) \in \mathcal{U}$, $(m_3, m_4) \in \mathcal{U}$ and $3 \leqslant s \leqslant t$, we have
\begin{align}
\nonumber s^{\kappa} f_{\kappa}(s) \geqslant&\ t^{\kappa} f_{\kappa}(t) - \frac{N^{4}_{5} - N^{3}_{5} + N^{2}_{5} - N^{1}_{5} + 1}{m_0 m_1 m_2 m_3 m_4} \int_{\frac{1}{t}}^{\frac{1}{s}} \frac{t^{\kappa} F_{\kappa}(t(1-x_1))}{x_1} d x_1 \\
\nonumber &+ \frac{2\left(N^{3}_{5} - 3 N^{2}_{5} + 7 N^{1}_{5} - 15\right)}{m_0 m_1 m_2 m_3 m_4} \int_{\frac{1}{t}}^{\frac{1}{s}} \int_{\frac{1}{t}}^{x_1} \frac{t^{\kappa} f_{\kappa}(t(1-x_1-x_2))}{x_1 x_2} d x_2 d x_1 \\
\nonumber &- \frac{6\left(N^{2}_{5} - 6 N^{1}_{5} + 25 \right)}{m_0 m_1 m_2 m_3 m_4} \int_{\frac{1}{t}}^{\frac{1}{s}} \int_{\frac{1}{t}}^{x_1} \int_{\frac{1}{t}}^{x_2} \frac{t^{\kappa} F_{\kappa}(t(1-x_1-x_2-x_3))}{x_1 x_2 x_3} d x_3 d x_2 d x_1 \\
\nonumber &+ \frac{24\left(N^{1}_{5} - 10 \right)}{m_0 m_1 m_2 m_3 m_4} \int_{\frac{1}{t}}^{\frac{1}{s}} \int_{\frac{1}{t}}^{x_1} \int_{\frac{1}{t}}^{x_2} \int_{\frac{1}{t}}^{x_3} \frac{t^{\kappa} f_{\kappa}(t(1-x_1-x_2-x_3-x_4))}{x_1 x_2 x_3 x_4} d x_4 d x_3 d x_2 d x_1 \\
\nonumber &- \frac{120}{m_0 m_1 m_2 m_3 m_4} \times \\
\nonumber & \quad \int_{\frac{1}{t}}^{\frac{1}{s}} \int_{\frac{1}{t}}^{x_1} \int_{\frac{1}{t}}^{x_2} \int_{\frac{1}{t}}^{x_3} \int_{\frac{1}{t}}^{x_4} \frac{t^{\kappa} F_{\kappa}(t(1-x_1-x_2-x_3-x_4-x_5))}{x_1 x_2 x_3 x_4 x_5} d x_5 d x_4 d x_3 d x_2 d x_1.
\end{align}
\end{corollary}

\section{Further prospect}
In this note, we only give some sieve inequalities and do not mention any possible application of these inequalities. In fact, these may be helpful in bounding the "sifting limits" $\beta_{\kappa}$ for $\kappa >1$. The bounds for $\beta_{\kappa}$ are quite important in many high--dimensional sieve problems. We hope someone can accomplish this work.

There are many other iteration rules proved in Brady's thesis \cite{BradyThesis}. We state two of them in the rest of this note, and we hope someone can generalize them.

\begin{theorem}
([\cite{BradyThesis}, Theorem 35]).
For any $w \leqslant z^2$, we have
\begin{align}
\nonumber S\left(\mathcal{A}, z\right) \geqslant&\ S\left(\mathcal{A}, w^{\frac{1}{2}}\right) - \sum_{w^{\frac{1}{2}} \leqslant p_1 < z} S\left(\mathcal{A}_{p_1}, \frac{w}{p_1}\right) + \frac{5}{6} \sum_{\frac{w}{p_1} \leqslant p_2 < p_1 < z} S\left(\mathcal{A}_{p_1 p_2}, \frac{w}{p_1}\right) \\
\nonumber &- \frac{2}{3} \sum_{\substack{\frac{w}{p_1} \leqslant p_3 < p_2 < p_1 < z \\ p_2 p_3 < w}} S\left(\mathcal{A}_{p_1 p_2 p_3}, \frac{w}{p_1}\right) -  \frac{1}{2} \sum_{\frac{w}{p_2} \leqslant p_3 < p_2 < p_1 < z} S\left(\mathcal{A}_{p_1 p_2 p_3}, \frac{w}{p_1}\right).
\end{align}
\end{theorem}

\begin{theorem}
([\cite{BradyThesis}, Theorem 42]).
If every element of $\mathcal{A}$ has size at most $y^{\frac{13}{12}}$ and $z^{\frac{12}{5}} < y < z^{\frac{5}{2}}$, we have
\begin{align}
\nonumber S\left(\mathcal{A}, z\right) \leqslant&\ S\left(\mathcal{A}, \frac{y}{z^2}\right) - \frac{4}{5} \sum_{\frac{y}{z^2} \leqslant p_1 < \frac{z^3}{y}} S\left(\mathcal{A}_{p_1}, \frac{y}{z^2}\right) - \frac{2}{3} \sum_{\frac{z^3}{y} \leqslant p_1 < \frac{y^2}{z^4}} S\left(\mathcal{A}_{p_1}, \frac{y}{z^2}\right) - \frac{8}{15} \sum_{\frac{y^2}{z^4} \leqslant p_1 < z} S\left(\mathcal{A}_{p_1}, \frac{y}{z^2}\right) \\
\nonumber &+ \frac{3}{5} \sum_{\frac{y}{z^2} \leqslant p_2 < p_1 < \frac{z^3}{y}} S\left(\mathcal{A}_{p_1 p_2}, \frac{y}{z^2}\right) + \frac{7}{15} \sum_{\frac{y}{z^2} \leqslant p_2 < \frac{z^3}{y} \leqslant p_1 < \frac{y^2}{z^4}} S\left(\mathcal{A}_{p_1 p_2}, \frac{y}{z^2}\right) \\
\nonumber &+ \frac{1}{3} \sum_{\frac{y}{z^2} \leqslant p_2 < \frac{z^3}{y} < \frac{y^2}{z^4} \leqslant p_1 < z} S\left(\mathcal{A}_{p_1 p_2}, \frac{y}{z^2}\right) + \frac{1}{3} \sum_{\frac{z^3}{y} \leqslant p_2 < p_1 < \frac{y^2}{z^4}} S\left(\mathcal{A}_{p_1 p_2}, \frac{y}{z^2}\right) \\
\nonumber &+ \frac{4}{15} \sum_{\frac{z^3}{y} \leqslant p_2 < \frac{y^2}{z^4} \leqslant p_1 < z} S\left(\mathcal{A}_{p_1 p_2}, \frac{y}{z^2}\right) + \frac{1}{5} \sum_{\frac{y^2}{z^4} \leqslant p_2 < p_1 < z} S\left(\mathcal{A}_{p_1 p_2}, \frac{y}{z^2}\right) \\
\nonumber &- \frac{2}{5} \sum_{\substack{\frac{y}{z^2} \leqslant p_3 < p_2 < p_1 < \frac{z^3}{y} \\ p_1 p_2 p_3^2 < z^2 }} S\left(\mathcal{A}_{p_1 p_2 p_3}, \frac{y}{z^2}\right) \\
\nonumber &- \frac{4}{15} \sum_{\frac{y}{z^2} \leqslant p_3 < p_2 < \frac{z^3}{y} \leqslant p_1 < \frac{y^2}{z^4} } \left(1-\frac{3\log(p_2 p_3)}{8\log(y / p_1)}\right) S\left(\mathcal{A}_{p_1 p_2 p_3}, \frac{y}{z^2}\right) \\
\nonumber &+ \frac{1}{5} \sum_{\substack{\frac{y}{z^2} \leqslant p_4 < p_3 < p_2 < p_1 < \frac{z^3}{y} \\ p_1 p_2 p_3^2 < z^2 }} S\left(\mathcal{A}_{p_1 p_2 p_3}, \frac{y}{z^2}\right) \\
\nonumber &+ \frac{1}{10} \sum_{\frac{y}{z^2} \leqslant p_4 < p_3 < p_2 < \frac{z^3}{y} \leqslant p_1 < \frac{y^2}{z^4} } \left(1-\frac{\log(p_2 p_3 p_4)}{\log(y / p_1)}\right) S\left(\mathcal{A}_{p_1 p_2 p_3}, \frac{y}{z^2}\right).
\end{align}
\end{theorem}

\bibliographystyle{plain}
\bibliography{bib}

\end{document}